\documentclass[12pt]{article} 

\usepackage[preprint]{tmlr}

\PassOptionsToPackage{hyphens}{url}
\usepackage{hyperref}
\hypersetup{
    colorlinks=true,
    linkcolor=blue,
    citecolor=purple,
    filecolor=magenta,      
    urlcolor=cyan,
    }
\usepackage{wrapfig}
\usepackage{lipsum}
\usepackage{booktabs}
\usepackage{mathtools}
\usepackage{amssymb,amsfonts,amsmath,graphicx}
\usepackage{amsthm}
\usepackage{thmtools} 
\usepackage{amsthm}
\usepackage[backend=biber,style=numeric-comp,sorting=nyt]{biblatex}

\usepackage{subfig} 

\usepackage{tikz-cd} 

\usepackage[normalem]{ulem}

\newcommand{\define}[1]{\textbf{#1}}

\newcommand{\VGT}{\text{VGT}}

\newtheorem{thm}{Theorem}
\newtheorem*{thm*}{Theorem}
\theoremstyle{definition}
\newtheorem{defn}{Definition}

\title{Exploring the Stratified Space Structure of an RL Game with the Volume Growth Transform}

\author{\name{Justin Curry} \email{jmcurry@albany.edu}\\
\addr University at Albany, SUNY
\AND
\name{Brennan Lagasse} \email{brennan.m.lagasse@lmco.com}\\
\addr Lockheed Martin Corporation
\AND
\name{Ngoc B. Lam} \email{ngoc.b.lam@lmco.com}\\
\addr Lockheed Martin Corporation
\AND
\name{Gregory Cox} 
\email{gecox@albany.edu}\\
\addr University at Albany, SUNY
\AND
\name{David Rosenbluth} \email{david.rosenbluth@lmco.com}\\
\addr Lockheed Martin Corporation
\AND
\name{Alberto Speranzon} \email{alberto.speranzon@lmco.com}\\
\addr Lockheed Martin Corporation
}

\begin{document}

\maketitle

\begin{abstract}
In this work, we explore the structure of the embedding space of a transformer model trained for playing a particular reinforcement learning (RL) game. 
Specifically, we investigate how a transformer-based Proximal Policy Optimization (PPO) model embeds visual inputs in a simple environment where an agent must collect ``coins'' while avoiding dynamic obstacles consisting of ``spotlights.'' 
By adapting Robinson et al.'s~\cite{robinson2024structure} study of the volume growth transform for LLMs to the RL setting, we find that the token embedding space for our visual coin collecting game is also not a manifold, and is better modeled as a stratified space, where local dimension can vary from point to point. 
We further strengthen Robinson's method by proving that fairly general volume growth curves can be realized by stratified spaces.
Finally, we carry out an analysis that suggests that 
as an RL agent acts, its latent representation alternates between periods of low local dimension, while following a fixed sub-strategy, and bursts of high local dimension,  where the agent achieves a sub-goal (e.g., collecting an object) or where the environmental complexity increases (e.g., more obstacles appear).
Consequently, our work suggests that the distribution of dimensions in a stratified latent space may provide a new geometric indicator of complexity for RL games. 
\end{abstract}

\section{Introduction}

It is a widely held belief that the success of neural networks to solve problems and understand relations between objects stems from the rich geometric structure of the latent spaces developed during training.
In this setting, the \emph{manifold hypothesis} asserts that tokenized representations of data are transformed to live on a low-dimensional manifold (the latent space) inside a high-dimensional feature space, where reasoning can occur.
However, recent work by \cite{robinson2024structure,robinson2025token} provides a strong rebuke to the manifold hypothesis in the context of large language models (LLMs), which suggests that tokens actually live on a \emph{stratified space} (see Appendix \ref{appx:sec:stratified-spaces}), where local dimension (Definition \ref{defn:local-dimension}) can vary from point to point.
In this work, we follow the methodology of \cite{robinson2025token} to interrogate the distribution of local dimensions in the first transformer block in a particular neural network architecture trained to solve a searing spotlight memory game introduced by \cite{pleines2025memory}.
In contrast to LLMs, this setting requires regarding images of game play as the fundamental (tokenized) representational unit.
Our findings indicate that the latent space structure for this particular RL game also exhibits non-manifold structure, suggesting that stratified spaces may be present in many different applications of neural networks.
It should be noted that the investigation of the dimension of latent representations of image-based data in deep learning is not new.
Indeed, \cite{popeintrinsic} calculated the so-called \emph{intrinsic dimension} of the MNIST, CIFAR-10, and ImageNet datasets to be between 26 and 43 dimensions. 
That work averaged the estimated local dimension over all points in the data set---an implicit deference to the manifold hypothesis.
By contrast, our work indicates that there are clusters of varying dimension, suggesting that such averaging procedures may be inappropriate, and gives a strong signal that the latent representation of images occurring in RL games may be stratified.
Most interestingly, we study the local dimension as a function of time---along trajectories of an RL agent during a single episode---and provide possible interpretations of fluctuation in local dimension as the agent navigates its environment.
Our analysis suggests that the latent space has higher local dimension near goal states (suggesting a correlation with reward), but also has higher local dimension near states where the RL agent appears to be unsure of how to act in states with high environmental complexity (suggesting a correlation with entropy of the learned policy). 

\section{The Volume Growth Transform}
\label{sec:methdology}

Recently, \cite{robinson2024structure} proposed a method for determining whether the token embedding space of large transformer models, such as GPT2, LLEMMA7B, and MISTARAL7B, forms a manifold. 
This approach differs substantially from other measures of manifoldness provided by \cite{TARDIS}, for example, but offers the advantage that it is fast to compute.

\begin{defn}\label{defn:VGT}
    Let $(X,d,\mu)$ be a metric measure space.
    The \define{volume growth transform} associates to each point $x\in X$ the log-log volume growth curve
    \[
    \VGT_x(s)\vcentcolon =\log \mu(B_x(r=e^s))\,.
    \]
    Here $\mu$ is a Borel measure on the space $X$ and $B_x(r)=\{y\in X \mid d(x,y)< r\}$ is the open ball of radius $r$ about $x$.
    Note that the radius is exponential in $s$.
\end{defn}

The analysis of \cite{robinson2024structure,robinson2025token} argues that certain tokens in LLMs cannot be viewed as samples from a manifold latent space by showing statistically significant departures from expected volume growth laws.
Indeed, as \cite[Equation 2]{robinson2024structure} observes---following the foundational work \cite[Theorem 3.1]{gray1974volume}---the volume of a small ball of radius $r$ around the point $x$ in a smooth $n$-manifold can be written as
\begin{align}
    v_x(r) = K r^{n} \left(1-\frac{Ric}{6(n+2)}r^2 + O(r^4)\right).
     \label{eq:vol_radium_relation}
\end{align}
Here $K$ is some scaling coefficient and $Ric$ is the local Ricci scalar curvature about $x\in \mathcal{T}$, where $\mathcal{T}\subset X$ is the embedded token space.
By taking the logarithm of both sides and considering a small $r$ approximation, we get:
\begin{align}
    \log v_x = \log K + n \log r - \frac{Ric}{6(n+2)}r^2 + O(r^4)\,.
    \label{eq:vol_radium_log_approx}
\end{align}
As we can see \eqref{eq:vol_radium_log_approx} is linear in $(\log K, n, Ric)$. 
In \cite{robinson2024structure}, the authors first estimate $\log K$ and $n$ and then recover the Ricci scalar curvature from the remaining residual. 
In this paper, we are primarily interested in estimating and interpreting the local dimension around a token. 
To this end, we set $Ric=0$ and add a subscript to $n$ to emphasize its dependence on $x$:
\begin{align}
    \log v_x = \log K + n_x \log r + O(r^4)\,. 
    \label{eq:vol_radium_log_approx_no_Ricci}
\end{align}
Note that \eqref{eq:vol_radium_log_approx_no_Ricci} gives the expected VGT for a Ricci-flat manifold of dimension $n$. 

Since we are studying the embedding of a finite collection of tokens under a transformer layer, we estimate volume simply by counting the number of tokens within radius $r$, i.e.
\begin{align}
    v_x(r) \approx M \left|\left\{y\in \mathcal{T} \mid d(x,y) \leq r\right\}\right|,
    \label{eq:vol_value}
\end{align} 
The metric $d$ used here is the Euclidean distance inherited from the ambient feature space $X\equiv\mathbb{R}^N$ that contains the embedded token space $\mathcal{T}$.
The constant $M$ is an unknown density parameter and can be combined with other constants, as discussed below.

\begin{figure}[t!]
    \centering
    \includegraphics[width=0.9\linewidth]{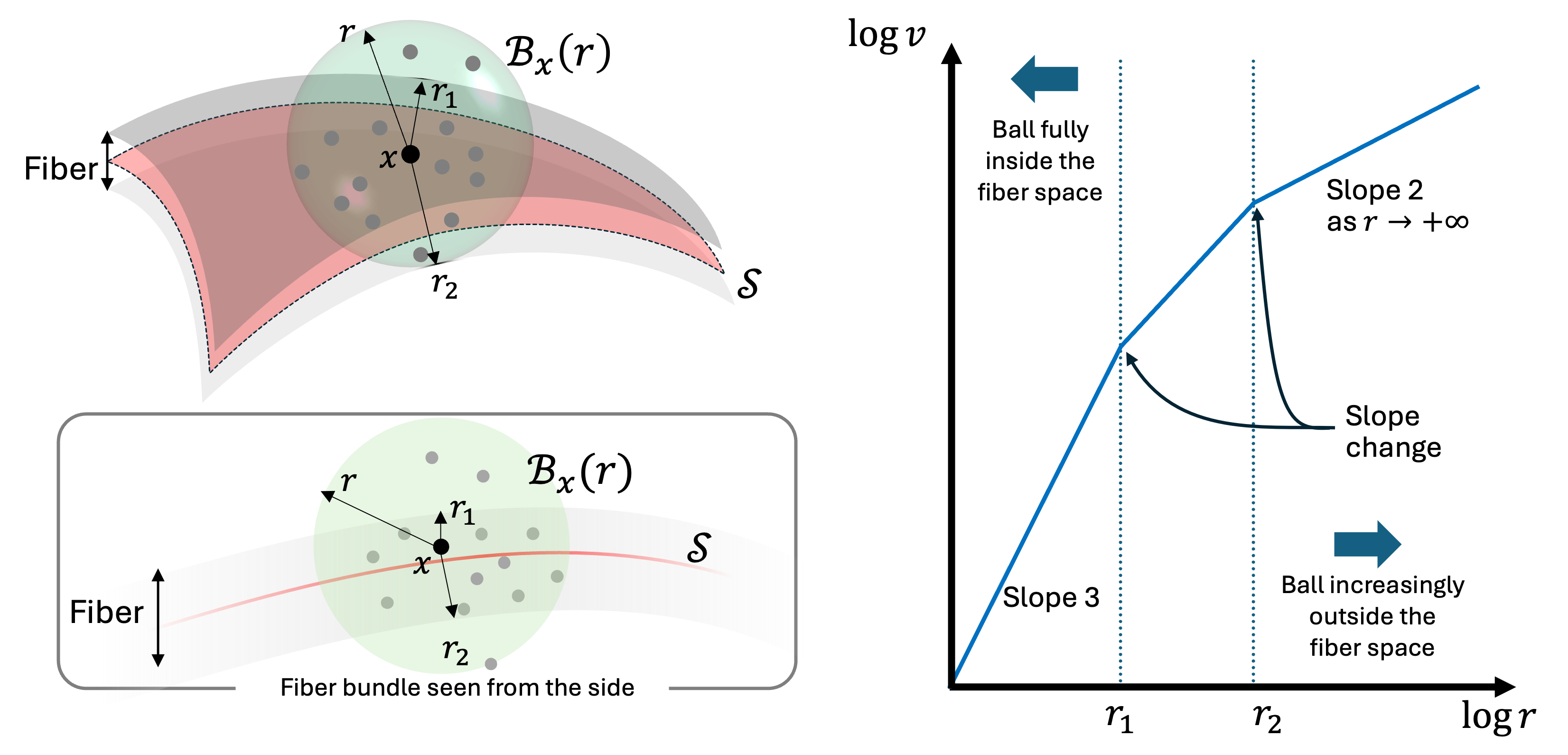}
    \caption{A cartoon example of the fiber bundle hypothesis and the corresponding piecewise-linear log-volume growth law. Above left the token space $\mathcal{T}$ is a 3-manifold with boundary, which admits a smooth retraction map $p:\mathcal{T}\to \mathcal{S}$ to a 2-manifold $\mathcal{S}$ in red. The fiber is a 1-dimensional gray interval $\mathcal{N}$, which models noise with bounded support. Growing a ball around token $j$ will lead to a piecewise linear growth curve with 3 distinct slopes: for small radii the volume will grow in accordance with a 3-manifold, then as the ball starts to ``exit'' the fiber growth will slow until it attenuates to the growth law for a 2-manifold.}
    \label{fig:fibered_space}
\end{figure}

Under the assumption that tokens are in general position around token $x$, no two tokens should appear at the exact same radius $r$.
This leads to an increasing sequence of radii $r_{1,x} < r_{2,x} < \cdots < r_{p,x}$, which indicates that there are $i$ tokens within distance $r_{i,x}$ from the token $x$. 
By substituting~\eqref{eq:vol_value} into~\eqref{eq:vol_radium_log_approx_no_Ricci}, we can estimate the local dimension $n$ by solving the following least squares problem, where we have set $\hat{H}_x \coloneqq \hat{K}_x/\hat{M}$:
\begin{align}
    \begin{pmatrix}
        0 \\
        \log 2\\
        \vdots\\
        \log p
    \end{pmatrix} \approx
    \begin{pmatrix}
        1 & \log r_{1,x}\\
        1 & \log r_{2,x}\\
        \vdots & \vdots\\
        1 & \log r_{p,x}
    \end{pmatrix}
    \begin{pmatrix}
        \log \hat{H}_x\\
        \hat{n}_x
    \end{pmatrix} + O(r^4)
    \label{eq:lsqt_problem}
\end{align}
Note that here $\hat{n}_x$ is the estimated value for the local dimension $n_x$ in~\eqref{eq:vol_radium_log_approx_no_Ricci}.
If Equation~\eqref{eq:lsqt_problem} admits a reasonable linear fit, then we have evidence that locally the latent space around token $x$ looks like a manifold and the value $\hat{n}_x$ gives the initial slope of $\VGT_x(s)$.

In \cite{robinson2025token} the authors suggest that certain classes of piecewise-linear volume growth curves are also plausible, and these can be modeled by a more general class of spaces called smooth fiber bundles.
The starting point for their analysis is that although some tokens may lie on a $m$-dimensional latent space $\mathcal{S}\subseteq \mathbb{R}^N$, there is also likely an orthogonal ``noise'' space $\mathcal{N}$ of dimension $n$ and that $m+n$ may not necessarily add up to the ambient embedding dimension $N$.
By setting a cutoff on normally distributed noise or assuming that noise has bounded support the \emph{fiber bundle hypothesis} suggests that the token space $\mathcal{T}$ is really a manifold \emph{with boundary}, which admits a smooth projection map $p:\mathcal{T}\to\mathcal{S}$, but where the fiber space $p^{-1}(s)\cong \mathcal{N}$ is independent of $s\in\mathcal{S}$.
Under this assumption, the log-volume growth about a token $x$ will exhibit decreases in slope, but with two significant regimes: small radii will show a growth rate in accordance with an $m+n$-dimensional manifold, and for large radii will attenuate to a growth law consistent with the base dimension $m$.

In summary, the methods of \cite{robinson2024structure,robinson2025token} point towards two avenues for rejecting the manifold hypothesis or the more general fiber bundle hypothesis, using the VGT:
\begin{enumerate}
    \item Study the distribution of local dimensions, defined using a specified range of radii, across all tokens $x$. If this distribution is not tightly clustered around a single integer $n$, as in Figure~\ref{fig:2coins_local_dim_hist}, then we can reject the manifold hypothesis.
    \item If for a single token $x$ we see a log-volume growth curve with a sharp \emph{increase} in growth, as in Figure~\ref{fig:2coins_vol_vs_rad} for token 4040, then this is inconsistent with the fiber bundle hypothesis and suggests that the token $x$ is lying on a lower-dimensional ``flare'' that is jutting off of a higher-dimensional bulk in the latent space.
\end{enumerate}
The most reasonable substitute for both the manifold and the fiber bundle hypothesis is the \emph{stratified space hypothesis}, which allows for a multi-modal distribution of local dimensions across tokens, as well as sharp increases in volume growth curve about single tokens. 
Analytical formulas for the volume-growth function of a stratified space are provided by generalizations of the Weyl Tube Lemma, cf.~\cite{brocker2000integral}, which allows for mixed polynomial growth for volume as a function of $r$, but the coefficients of these polynomials rely on hard-to-estimate Lipschitz-Killing measures, and are beyond the scope of this work.
Although no method currently exists for characterizing precisely the stratified space structure of a noisy point-cloud sample in terms of volume-growth laws, we do provide a realization theorem, which shows that fairly arbitrary volume growth curves can be modeled by a stratified space: 

\begin{thm}[Realization of the VGT at a Point]\label{thm:realization-of-VGTx}
    If $f:[0,\infty)\to[0,\infty)$ is a non-decreasing piecewise linear function of the form
    \begin{itemize}
        \item $f|_{[0,s_1]}=n_1\cdot s$,
        \item $f|_{[s_i,s_{i+1}]}=n_{i+1}\cdot s + f(s_i)$ for $i\in \{1,\ldots,k-1\}$, and
        \item $f|_{[s_k,\infty)}=f(s_k)$ is a constant,
    \end{itemize}
    for some finite set of ``critical scales'' $\{s_1 < \cdots < s_k\}$ and natural number slopes $n_1,\ldots, n_k\in \mathbb{N}$, then there exists a stratified space $X$ with a point $x\in X$ such that
    \[
        \VGT_x(s)=f(s).
    \]
\end{thm}

The proof of Theorem \ref{thm:realization-of-VGTx} is deferred to Appendix \ref{appx:sec:stratified-spaces}.

\section{Experimental Setup and Numerical Results}

As stated in the introduction, the experimental work of \cite{robinson2024structure,robinson2025token} is concerned entirely with the latent space structure of open-sourced large language models.
One of the central contributions of our work is the application of their approach to the new context of RL where tokens are not textual, but rather images of the environment an agent acts on, and where the token embedding space is organized via a reward function shaped so that the agent solves a given task.
We find that many of the same patterns observed in \cite{robinson2024structure,robinson2025token} in text data are also observed in image data gathered from the \emph{Searing Spotlight} game of the Memory Gym environment developed by \cite{pleines2025memory}, which is freely available at \cite{memory_gym_github}.

For our study, a token is no longer a text symbol or a word fragment, but rather an 84\texttimes84 colored image collected during game play; see Figure \ref{fig:games}.
In the original Searing Spotlights game, a checkerboard room with a single ``coin'' is initially revealed in full light.
The lights then dim to complete darkness and the agent must navigate to the coin using memory alone, while also avoiding spotlights.
If the agent is caught in a spotlight, it then loses health for each time-step it spends in the spotlight's field of view. 
If sufficient health is lost before collecting the coin, the agent dies and the trial is unsuccessful.
The reward function, detailed in Table~\ref{table:env_params} in Appendix~\ref{appx:sec:experiments}, gives reward 1 for collecting a coin, reward -0.05 for every time-step in the spotlight, and reward -0.005 reward for every time-step. 

\begin{wrapfigure}[19]{r}{0.55\textwidth}
    \centering
    \vspace*{-3ex}
    \includegraphics[width=\linewidth]{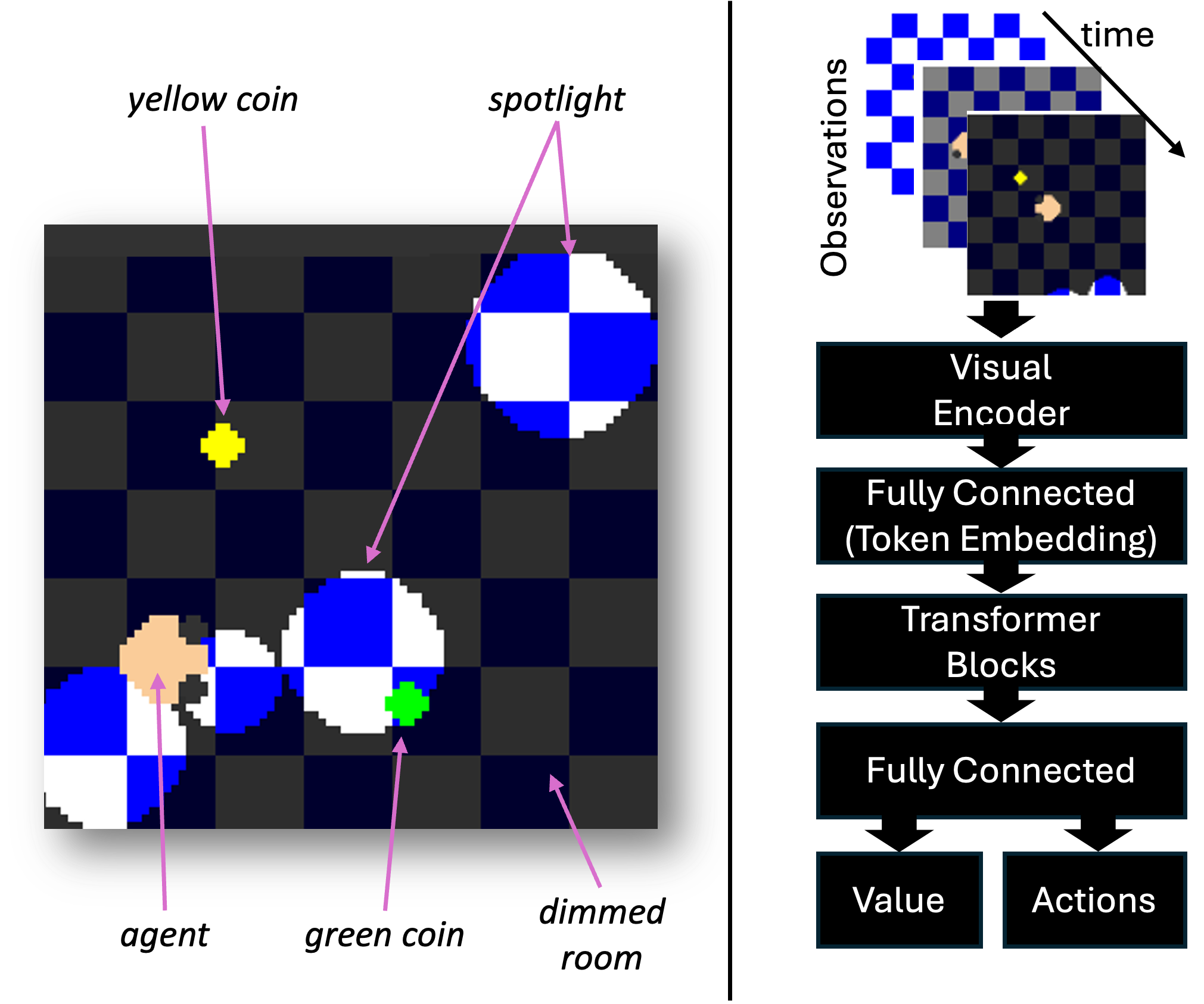}
    \vspace*{-3ex}
    \caption{(Left): The `Two-Coin' game. (Right): The network architecture used to train the agent.}
    \label{fig:games}
\end{wrapfigure} 

We modified the Searing Spotlights environment to have two coins in the room, but with different colors.
This allows us to look for patterns in the local dimension curve and try to connect signals there with repeatable actions during game play, e.g., coin collection. 
We also increased visibility of the room, so that we could more readily interpret the movement of the agent during each episode in relation to its environment, cf.~Section~\ref{subsec:loc-dim-traj}. 
Furthermore, we removed a previously visible ``health bar'' and ``action bar'' so that the visual input to our neural net architecture did not include this information.
We trained the agent using a network architecture similar to the one described in \cite{pleines2025memory}, which is based on the PPO algorithm. 
Our architecture utilizes two layers of transformer blocks, enabling the agent to learn effective policies for the given tasks. 
A sketch of the architecture is shown on the right of Figure~\ref{fig:games}. 
We simulated 250 runs with the agent starting from different initial conditions. 
On average, the agent completed the task in 18 steps, resulting in approximately 4500 observations or tokens. 
Further details about the environments, reward functions, network architecture, and learning curves, are reported in Appendix~\ref{appx:sec:experiments}.

\subsection{Local Dimension via Volume Calculations}

We computed the local dimension for each token using the least squares problem in Equation~\eqref{eq:lsqt_problem} in two different ways. The first approach involves fixing an interval of possible radii, considering the tokens within that interval, and using that information to estimate the corresponding local dimension. 
The results of this approach are shown on the left of Figure~\ref{fig:2coins_local_dim_hist} for radii between 40 and 60. 
The advantage of this approach is that it allows us to control the maximum radius, thereby better handling some of the approximations in the formulas presented in Section~\ref{sec:methdology}. 
However, in general, we would need to choose an interval of radii that is adapted to each token to balance approximation error and ensuring a sufficient number of tokens to obtain reliable estimates through the least squares problem. 
The resulting histogram shows local dimensions ranging from 6 to 45. 
Note that a subset of tokens have ``zero'' dimension, indicating that there were insufficient neighboring tokens to obtain a least squares solution. 
The second approach, followed by~\cite{robinson2024structure}, considers the data within a given interval of volume and does not bound the radius. 
In this case, the estimate of local dimension appears to be more consistent and virtually no tokens have local dimension higher than 21. 
Note that the first part of the two distributions of Figure~\ref{fig:2coins_local_dim_hist} are consistent---see the orange box for each. 
Despite the fact that the exact probability density values are different, the shape of the distribution is similar, showing a large cluster of tokens having dimension 6 to 10. 
Based on an inspection of the local dimension plots in Figure~\ref{fig:2coins_vol_vs_rad}, we selected the volume range of 50 to 90, as that indicated the most linear behavior across tokens. 
As a result of all these considerations, we use this second ``volume range'' approach to estimate local dimensions for the remainder of the paper.

\begin{figure}[t]
    \centering
    \includegraphics[width=0.95\linewidth]{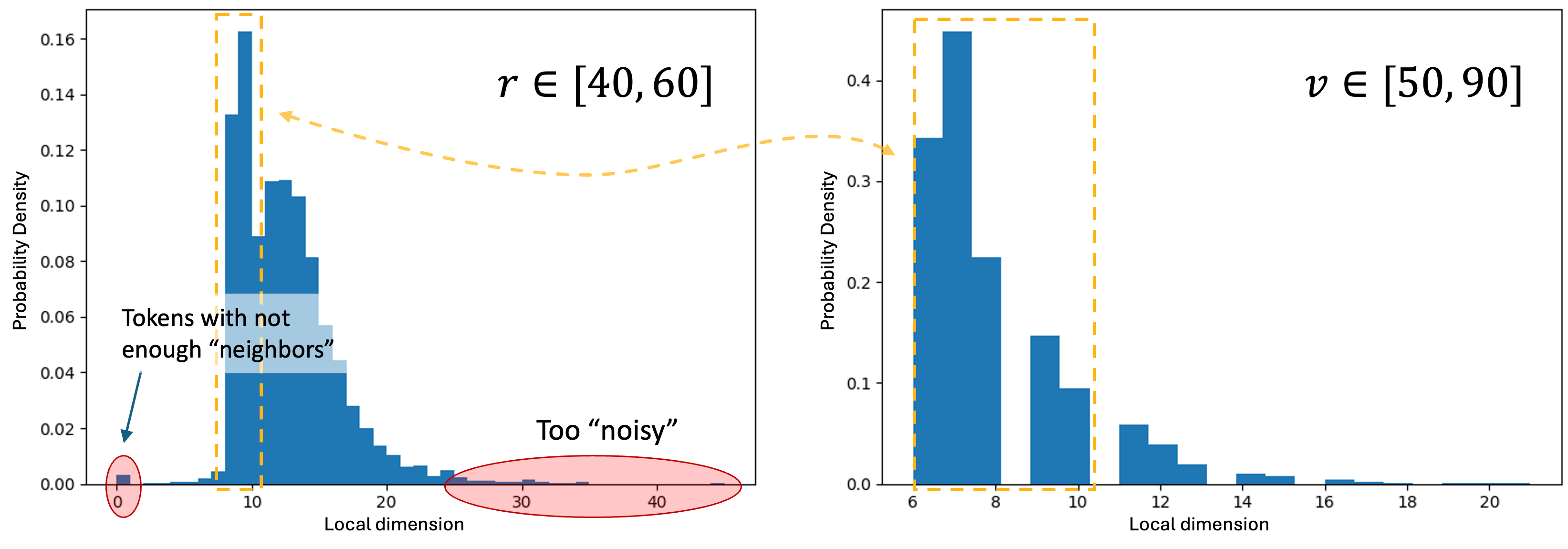}
    \caption{Local dimension probability density estimate when solving the least square problem~\eqref{eq:lsqt_problem} with the radius being bounded in the interval [40, 60] (Left) or the volume being bounded in the interval [50, 90]. Note that the behavior on the left plot is similar to the right, namely with the highest density in the range [6,10].}
    \label{fig:2coins_local_dim_hist}
\end{figure}

\begin{figure}[ht]
    \centering
    \includegraphics[width=\linewidth]{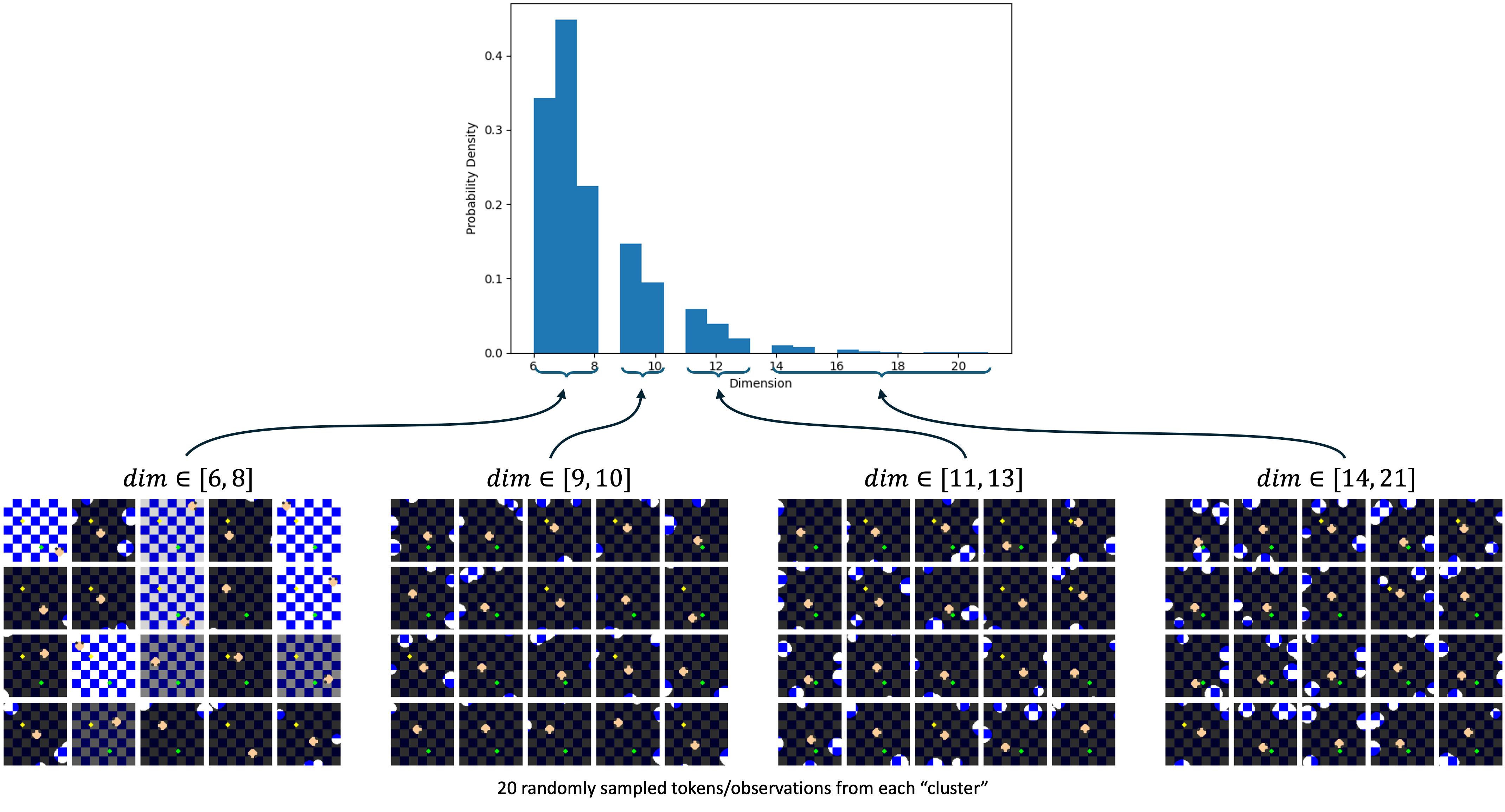}
    \caption{For each local dimension ``cluster'' we show 20 tokens (observations). Note that tokens with low dimension (6 to 10) are associated to simple situations, with not too many spotlights. Tokens with high dimension local dimension (14 to 21) are associated to situations with more spotlights, requiring more complex encoding.}
    \label{fig:2coins_loc_dim_cluster}
\end{figure}

Figure~\ref{fig:2coins_local_dim_hist} reveals four distinct clusters of tokens, with local dimensions in the range of [6,8], [9,10], [11,13], and [14,21]. 
To gain insight into the nature of these clusters, we took 20 random samples from each cluster and plotted them in Figure~\ref{fig:2coins_loc_dim_cluster}. 
An initial analysis reveals some intriguing trends. 
Notably, tokens with lower local dimension appear to be associated with the start of game play, when the lights are still on in the room, or recently after they were dimmed and not many spotlights are present.
In contrast, observations with high local dimensions are often dense with spotlights. 
This observation is intuitively plausible, as we would expect tokens representing complex situations to belong to strata with sufficient degrees of freedom to effectively capture and track the multiple spotlights. 

In Figure~\ref{fig:2coins_vol_vs_rad} we plot the volume growth curves for eight tokens, with two tokens randomly selected from each of the four clusters in Figure~\ref{fig:2coins_loc_dim_cluster}.
Several interesting trends emerge from this plot. 
For certain tokens, such as 2434, 2833, 4040, and 4105, we observe that the slopes at the corner points---where the slope changes---do not conform to the fiber bundle hypothesis outlined in Section~\ref{sec:methdology} and illustrated in Figure~\ref{fig:fibered_space}. 
Specifically, we note that after an initial corner where the slope decreases, it subsequently increases again, thereby violating the fiber bundle hypothesis and, by extension, the manifold hypothesis. 
This suggests that these tokens reside on ``flares'' of a stratified space. 
In contrast, for other tokens, such as 136, 1009, and 3758, the slope changes are much more subtle, if present at all. 
This implies that these tokens are likely to be situated in the bulk of the latent space. 
For the remaining tokens, the slope changes are more nuanced, making it challenging to argue precisely what local stratification structure is present there.

\begin{figure}[t]
    \centering
    \includegraphics[width=0.90\linewidth]{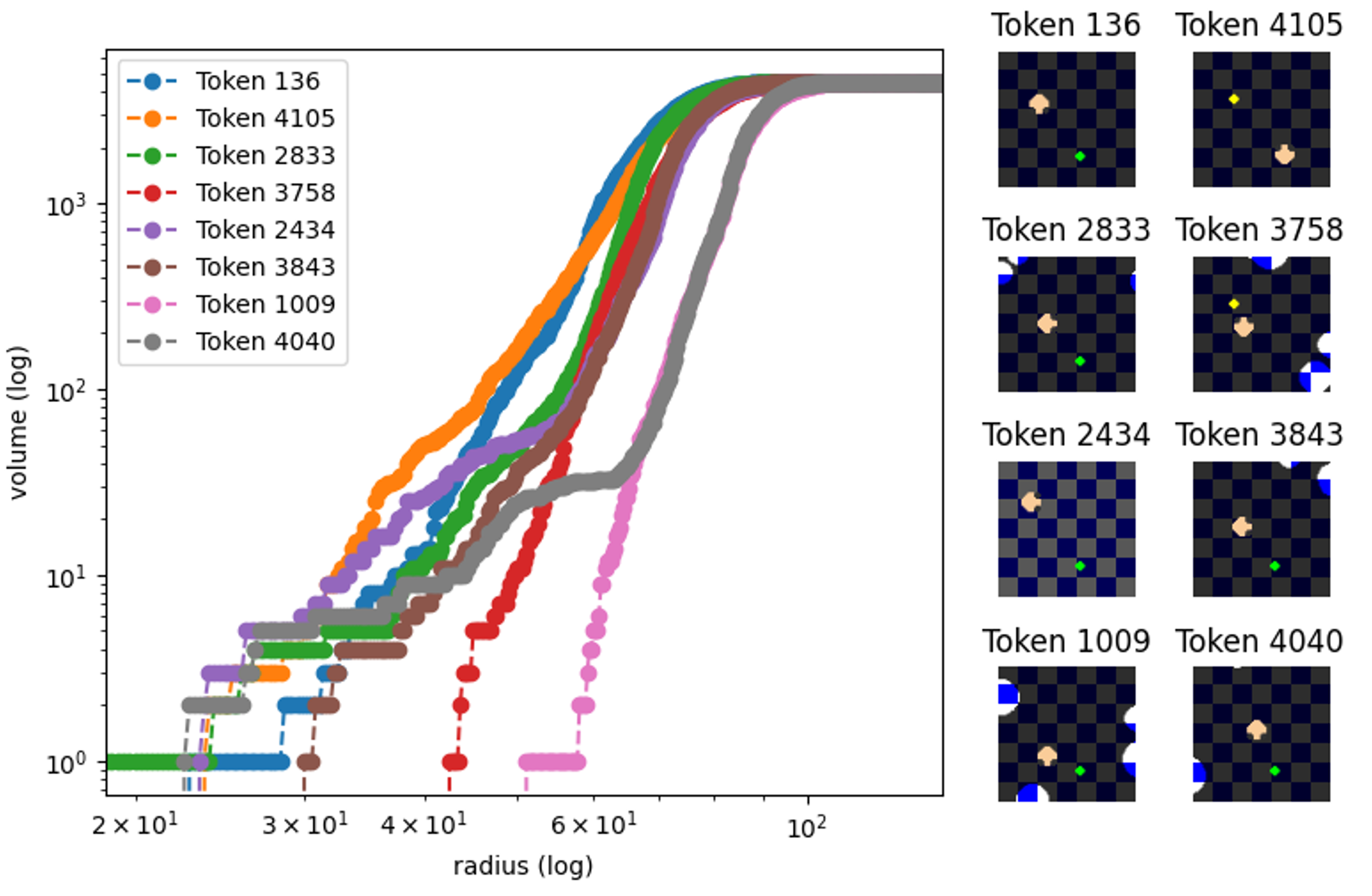}
    \caption{Volume vs radius for various tokens/observations. Note that for tokens 2434, 2833, 4040, and 4105 the slopes at the corner points, where the slope changes, do not conform to the fiber bundle hypothesis and thus with the manifold hypothesis.}
    \label{fig:2coins_vol_vs_rad}
\end{figure}

\subsection{Local Dimension Function Along a Trajectory}
\label{subsec:loc-dim-traj}

We also perform an analysis not conducted by \cite{robinson2024structure,robinson2025token}, which leverages the fact that in a RL setting we have access to full-length policies, which consist of state-action sequences. 
This allows us to examine the evolution of the local dimension along trajectories and form possible interpretations of this dimension in terms of the game play. 
The results of this analysis are presented in Figures~\ref{fig:sim1} and~\ref{fig:sim2} and discussed below. 

In both Figure~\ref{fig:sim1} and Figure~\ref{fig:sim2} we have highlighted regions where the local dimension function spikes.
We note that in both figures these spikes occur around the time of coin collection.
Additionally, in both trajectories the second spike is higher than the first.
We explain this by observing that in both simulations spotlights are starting to enter the environment, which requires a more complex modeling space for the agent to avoid them. 

However, we find that the number of options for the agent to avoid a spotlight seems to be a more significant driver of dimension than the spotlights themselves.
Indeed, in Figure~\ref{fig:sim1} we see a global maximum in dimension at time $t=12$, when a second spotlight is entering the frame at the bottom.
The agent is now torn between two options: head towards the coin or avoid the spotlight that it is headed towards?
As the agent moves right in time steps 13 and 14, we see that the agent has settled on moving towards the coin, even though it increases proximity to the bottom spotlight.
The last spike in local dimension, at $t=15$, corresponds to a frame where the agent momentarily looks west, perhaps weighing the possibility of fleeing again, before collecting the coin at $t=16$.

\begin{figure}[th]
    \centering
    \subfloat[]{\includegraphics[width=0.98\linewidth]{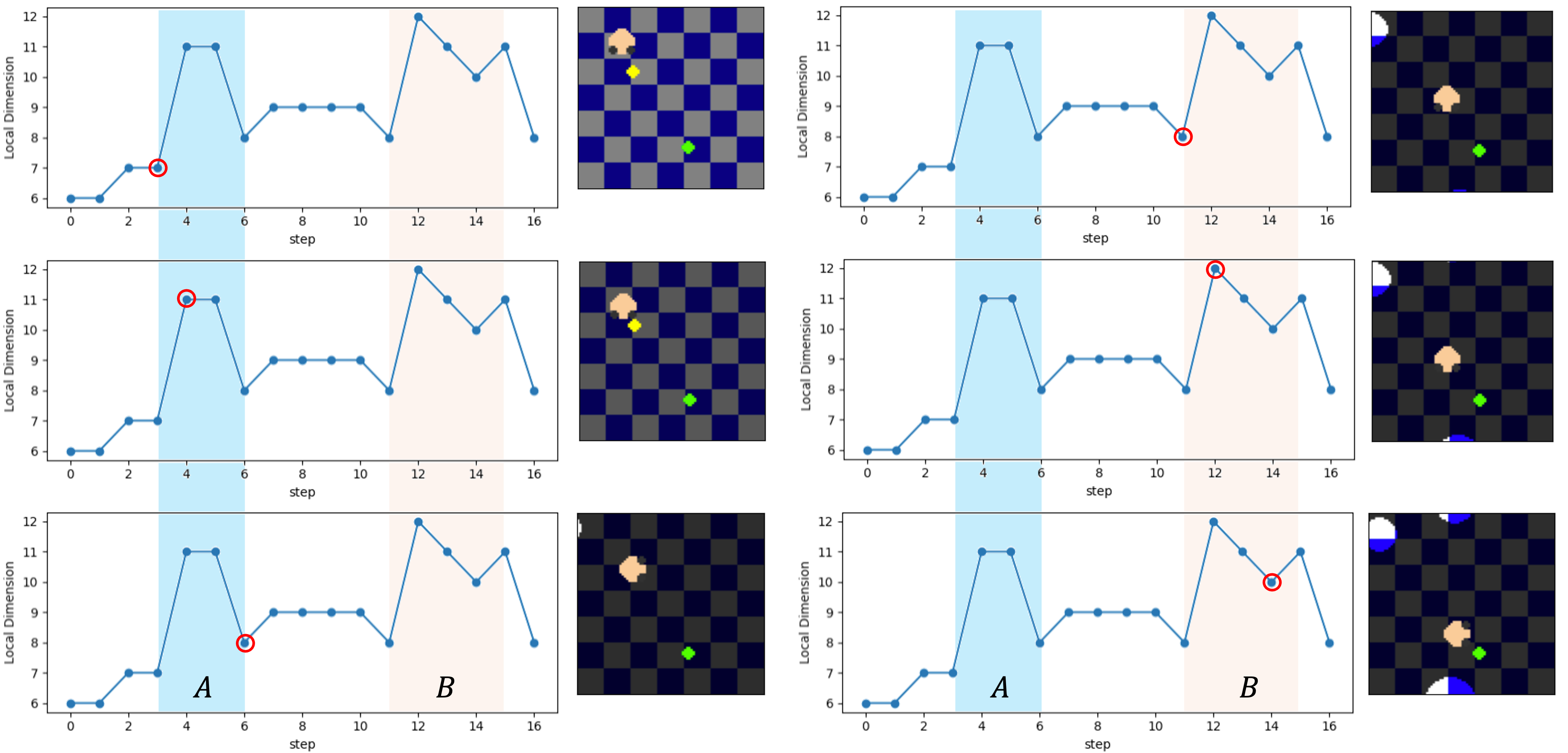}\label{fig:sim1}}\\[2ex]
    \subfloat[]{\includegraphics[width=0.98\linewidth]{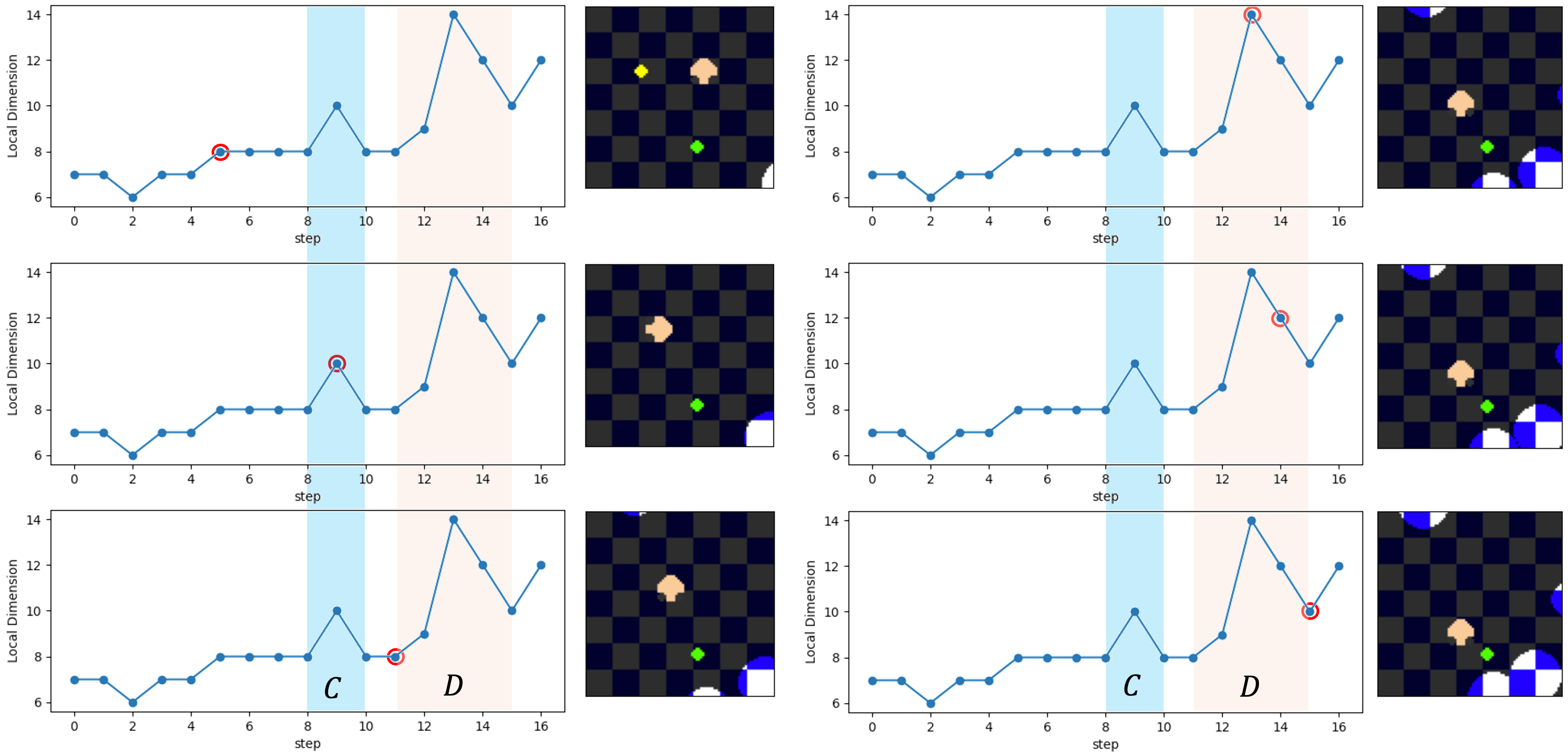}\label{fig:sim2}}
    \caption{Two simulation traces, (a) and (b), showing how the local dimension of the tokens (observations) embeddings change as the agent moves in the environment. Time is increasing from top to bottom and left to right.}
    \label{fig:2coins_sim}
\end{figure}

A similar analysis appears to hold for the trajectory in Figure~\ref{fig:sim2}.
The agent has just collected the first, yellow coin at $t=9$, just as as a spotlight from the bottom right-hand corner gains increased prominence.
The agent again is looking away from the second coin, as if deciding which policy to pursue.
As the agent commits to pursuing the green coin again, the dimension drops, despite the appearance of a second spotlight.
The maximum dimension at $t=13$ does correspond to a frame where there are four distinct spotlights, but the dimension does drop again when the bottom two spotlights merge into one, thus creating one effective spotlight.
We note that this last observation corresponds nicely with the stratified space example in Appendix~\ref{appx:sec:stratified-spaces}.

\section{Conclusion}

In conclusion, our study investigates the geometric structure of the embedding space of a transformer model trained to play a particular reinforcement learning game; a modification of the Searing Spotlights environment developed by~\cite{pleines2025memory}.
We find, in agreement with similar analyses carried out by~\cite{robinson2024structure,robinson2025token} for large language models, that the token embedding space learned by the model does not form a manifold or a fiber bundle, but rather a more complex geometric structure known as a stratified space. 
Based on our analysis of the local dimension function of the token embedding space along a trajectory, we note that spikes in the local dimension function generally occur before the agent reaches a goal state, where a coin can be collected.
We find that high-dimensional tokens generally correspond to frames during game play where there are either lots of spotlights or when the agent is appearing to weigh multiple possible courses of action---either avoiding the spotlights or pursuing the coins.
This interpretation coheres with the analysis of word embeddings presented by~\cite{jakubowski2020topology}, which suggests that words with multiple distinct meanings---a phenomenon known as \emph{polysemy}---tend to correspond to non-manifold points in these word embedding latent spaces.
We conjecture that distinct strata in the latent space correspond to different state-action trajectories, with increases in local dimension occurring when the agent is in a region where strata intersect. 
 This stratified structure may provide an automatic way to map distinct strata onto compressed symbolic representations of state-action trajectories in RL games, with relationships between strata understood in terms of syntactic properties that govern how an agent might assemble these trajectories into larger strategies.  
Finally, we note that since tokens associated with high local dimension indicate situations that the agent perceives as more complex, an adaptive training procedure could make use of this information to select training examples to help the agent resolve future complex scenarios.


\medskip

\printbibliography

\appendix

\section{Manifolds, Fiber Bundles, and Stratified Spaces}
\label{appx:sec:stratified-spaces}

We provide a brief introduction to several classes of spaces that can be used to model the latent space of a transformer. 
A subset $M\subseteq \mathbb{R}^N$ of an embedding space is a \define{$m$-dimensional manifold} if every point $p\in M$ belongs to an open set $U_p$ such that $U_p\cap M$ is homeomorphic to an open subset $V\subseteq \mathbb{R}^m$ \cite{guillemin2010differential}.
By applying a translation, one can assume that this homeomorphism $h_v:U_p\cap M \to\mathbb{R}^m$, called a \define{chart}, sends the point $p$ to the point $0\in\mathbb{R}^m$.
The invariance of dimension theorem implies that $U_p\cap M$ and $V$ both must have locally constant dimension $m$ and thus, if $M$ is a connected subset, it must have a single fixed dimension $m$; cf. page 48 of~\cite{tu2011manifolds}.

Stratified spaces, by contrast, can have points of varying local dimension. 
Historically, stratified spaces were introduced by~\cite{whitney1957elementary} to provide a convenient category of spaces for studying algebraic varieties, which are equivalently called zero sets.
A \define{zero set} is any set of the form $X=\{x\in \mathbb{R}^N \mid P(x)=0\}$ for some polynomial $P$. 
Varieties are not, in general, manifolds, but can be unions of manifolds of different dimensions. 
For a first example, the polynomial $P(x,y,z)=x^2+y^2-z^2$ has a zero set that is the union of two infinite cones glued together at the origin, which is a point that breaks the definition of a 2-manifold written above.
For another example, the polynomial $Q(x,y,z)=z(x^2+y^2)$ has a zero set that is the union of the $x,y$ plane (where $z=0$) along with the $z$ axis (where $x=y=0$).
Notice that this example has a 1-dimensional portion (the $z$ axis) as well as a 2-dimensional portion (the $x,y$ plane).
Whitney proposed a general process for decomposing varieties via a filtration
\[
\varnothing = X_{-1} \subseteq X_0 \subseteq \cdots \subseteq X_N = X
\]
where, for each $i\in \{0,\ldots,N\}$, the set $X^i:=X_i\setminus X_{i-1}$ is a disjoint union of $i$-dimensional manifolds. 
This notion turned out to be applicable for spaces more general than zero-sets of polynomials. 
For example, if $f:M\to N$ is a smooth mapping between manifolds, then one can define a similar filtration of $M$ using the rank of the derivative $df$. A case where this filtration is largely trivial is given by \define{fiber bundles}, which is any surjective smooth map $f:M \to N$ between manifolds where the derivative $df_p$ has constant full rank.
The Pre-Image Theorem (or Fiber Lemma) asserts that $f^{-1}(q)$, which is called the \define{fiber} over $q\in N$, is always an $(m-n)$-dimensional manifold, cf.~page 21 of~\cite{guillemin2010differential}.

The study of stratified spaces has evolved substantially since the 1950s, with many competing variations and generalizations---see~\cite{friedman2020singular} for a modern survey. 
Our own analysis of the stratified space structure of RL games is more natural from the poset stratification perspective, which first appeared in \cite{lurie17higher} and has been studied in more detail by \cite{ayala2017local} and \cite{waas2024stratifications}, among others.
In the language of \cite{lurie17higher}, a \define{$\mathcal{P}$-stratified space} is any topological space $X$ along with a continuous map $\sigma: X \to (\mathcal{P},\leq)$ where $\mathcal{P}$ is a poset equipped with the Alexandrov topology, that is, the topology where down-sets\footnote{A down-set $D\subseteq (\mathcal{P},\leq)$ is a subset where $y\in D$ and $x\leq y$ implies $x\in D$.} are closed.
We note that when $(\mathcal{P},\leq)=\mathbf{[d]}:=0\leq 1\cdots\leq d$ is the totally ordered set with $d+1$ elements, then one obtains the filtration of $X$ considered above.
For a more suggestive example that connects with our Searing Spotlights environment, consider the configuration space of $k$ distinct points in the plane, $\mathbb{R}^2$.
The topological space that models all possible locations of these $k$ points is $\mathbb{R}^{2k}$, but this space admits a $\mathbf{[k]}$-stratification $\sigma:\mathbb{R}^{2k}\to \mathbf{[k]}$ by counting the number of points with distinct locations.
For example, if all $k$ points are ``piled up'' at the origin, which is encoded by the vector $[(0,0),\ldots,(0,0)]\in \mathbb{R}^{2k}$, then this configuration would map to the element $1\in \mathbf{[k]}$ via $\sigma$.
Alternatively, configurations where all $k$ points (or spotlights) have distinct locations would belong to the highest-dimensional stratum, i.e., $\sigma^{-1}(k)=\mathrm{Conf}_k(\mathbb{R}^2)\subseteq \mathbb{R}^{2k}$, which is the ordered configuration space of $k$ points.

Strictly speaking, poset stratified spaces need not have fibers that are manifolds, but this can be achieved by asking for an extra \define{locally $C^0$-stratified} condition.
Following a combination of Section 2 of \cite{ayala2017local} and
Definition A.5.5 of \cite{lurie17higher}, one typically asks that a $\mathcal{P}$-stratified space have the additional property that each $p\in \mathcal{P}$ and $x\in \sigma^{-1}(p)$ has a corresponding compact $\mathcal{P}_{>a}$-stratified space $Y$ and a filtration preserving embedding $\mathbb{R}^i\times C(Y)\hookrightarrow X$ whose image $U_x$ contains $x$. 
Here $C(Y)$ corresponds to the cone on $Y$, which promotes the $\mathcal{P}_{>a}$-stratified space $Y$ to the $P_{\geq a}$-stratified space $C(Y)$ where the cone point maps to the element $a\in \mathcal{P}$.
This is all cleanly expressed with the following commutative diagram, where the dimension of $\mathbb{R}^i$ can vary as a function of $x$:
\[
\begin{tikzcd}
 \mathbb{R}^i\times C(Y) \arrow[r, hookrightarrow, "h_x"] \arrow[d] & X \arrow[d, "\sigma"] \\
 \mathcal{P}_{\geq p} \arrow[r, hookrightarrow] & \mathcal{P}
\end{tikzcd}
\]
Recent work \cite{nocera2023whitney} proves that every Whitney stratified space is a conically smooth stratified space in the sense of \cite{ayala2017local}, so this new theory includes the classical notion of stratification introduced by Whitney.

The dimension $i$ of the stratum containing $x$ in the diagram above is best thought of as the \define{intrinsic dimension} of the containing stratum, and is not what is typically meant by \emph{local} dimension. 
We now provide a precise definition of this concept, following \cite{hironaka1986local} and \cite[\S 2.4]{ayala2017local}

\begin{defn}\label{defn:local-dimension}
    Suppose $\sigma: X \to \mathcal{P}$ is a $C^0$ stratified space where $\mathcal{P}$ is a finite poset.
    The \define{local dimension} at a point $x\in X$ is the maximum (covering) dimension of any connected stratum $S \subseteq X^p:=\sigma^{-1}(p)$ such that $x\in \overline{S}$ is in the closure of $S$.
\end{defn}

Finally, we re-state and prove Theorem \ref{thm:realization-of-VGTx}.

\begin{thm*}[Realization of the VGT at a Point]\label{thm:VGT-appendix}
    If $f:[0,\infty)\to[0,\infty)$ is a non-decreasing piecewise linear function of the form
    \begin{itemize}
        \item $f|_{[0,s_1]}=n_1\cdot s$,
        \item $f|_{[s_i,s_{i+1}]}=n_{i+1}\cdot s + f(s_i)$ for $i\in \{1,\ldots,k-1\}$, and
        \item $f|_{[s_k,\infty)}=f(s_k)$ is a constant,
    \end{itemize}
    for some finite set of ``critical scales'' $\{s_1 < \cdots < s_k\}$ and natural number slopes $n_1,\ldots, n_k\in \mathbb{N}$, then there exists a stratified space $X$ with a point $x\in X$ such that
    \[
        \VGT_x(s)=f(s).
    \]    
\end{thm*}

\begin{proof}
    Let $X^{n}_{r}=\mathbb{D}^{n}_r \cap \mathbb{H}^n$ denote the intersection of the closed disc of radius $r$ in dimension $n$ with the upper half space $\mathbb{H}^n:=\{x\in \mathbb{R}^n \mid x_n\geq 0\}$.
    Let $z^n_{r}=(0,\ldots,r)\in X^n_r$ denote the ``zenith'' of the closed half-disc of radius $r$ and let $c^{n}=(0,\ldots,0)\in X^n_r$ denote the ``center'' (zero vector) of the half-disc.
    Setting $r_i=e^{s_i}$, then the space
    \[
    X:= X^{n_1}_{r_1} \underset{\sim}{\sqcup} X^{n_2}_{r_2} \underset{\sim}{\sqcup} \cdots  \underset{\sim}{\sqcup} X^{n_k}_{r_k}
    \] 
    where the zenith $z^{n_i}_{r_i}$ is identified with the center $c^{n_{i+1}}$ via the equivalence relation $\sim$.
    The log-volume growth curve for a ball centered around $c^{n_1}$ will produce the function $f$.

    It is routine to check that this space is stratified.
    Following Whitney's version of stratified space we place the collection of equivalence classes $[z^{n_i}_{r_i}]=[c^{n_{i+1}}]$ in the set of zero-dimensional strata. 
    The boundary of each half-disc with these points removed $(\partial X^{n_i}_{r_i})-z^{n_i}_{r_i}$ are naturally connected $(n_i-1)$-dimensional strata. 
    Finally, the interior $\mathrm{int}(X^{n_i}_{r_i})$ will be pure connected $n_i$-dimensional strata as well.
\end{proof}

\section{Experiment Setup}
\label{appx:sec:experiments}

\subsection{Model}
We used the Transformer-XL-based PPO model implemented in \cite{memory_gym_github} to train agents for the Searing Spotlights environments outlined below. At each step interacting with an environment, the model receives a 84x84 RGB grid as an input. This input is embedded by a CNN as a 256 dimensional vector which is then passed through a Transformer-XL and policy/value heads. The model architecture is visualized in Figure \ref{fig:model_architecture}. The hyperparameters for the model used across experiments are listed in Table \ref{table:model_params}.

\begin{figure}
    \centering
    \includegraphics[width=0.65\textwidth]{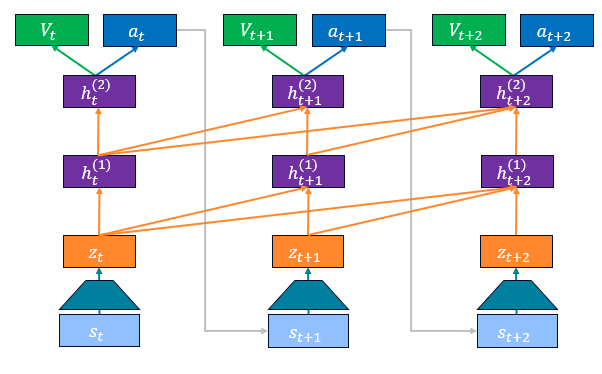}
    \caption{Illustration of the Transformer-XL-based PPO model implemented in \cite{memory_gym_github}. State $s_t$ is the visual observation of the agent. Embedding $z_t$ is computed from state $s_t$ using a CNN (cyan). Hidden states $h^{(i)}_{t}$ are computed by attending over hidden states from the previous layer as outlined in \cite{dai2019transformerxlattentivelanguagemodels} (depicted by the orange arrows). The final hidden state is the input for the value and policy heads. Training is performed with the PPO loss function.}
    \label{fig:model_architecture}
\end{figure}

\begin{table}[ht]
    \centering
    \begin{tabular}{p{3in}l}
    \toprule
    \textbf{Parameter} & \textbf{Value} \\
    \midrule
    Number of TrXL Blocks & 2 \\
    Embedding Dimension & 256 \\
    Number of Attention Heads & 4 \\
    Memory Window Length & 16 \\
    Block weight initialization & Xavier \\
    \bottomrule
    \end{tabular}
    \caption{Model parameters.}
    \label{table:model_params}
\end{table}

\subsection{Environment: The Two-Coins Game}

In the two-coins environment, the agent needs to collect two coins while evading spotlights that travel across the screen. The game ends when the agent collects both coins, runs out of time, or loses all health from contact with the spotlights. The agent is rewarded for collecting coins, and penalized for taking damage and/or taking too long to collect the coins. The agent spawns in a random spot on the grid and the coins spawn at the same location each episode. The parameters for the environment, as introduced in \cite{pleines2025memory}, are outlined in Table \ref{table:env_params} and the hyperparameters for the model trained on this environment are outlined in Table \ref{table:model_params}. The success rate of the agent over the course of training is shown in Figure \ref{fig:train_results}.

\begin{figure}
    \centering
    \includegraphics[width=0.55\textwidth]{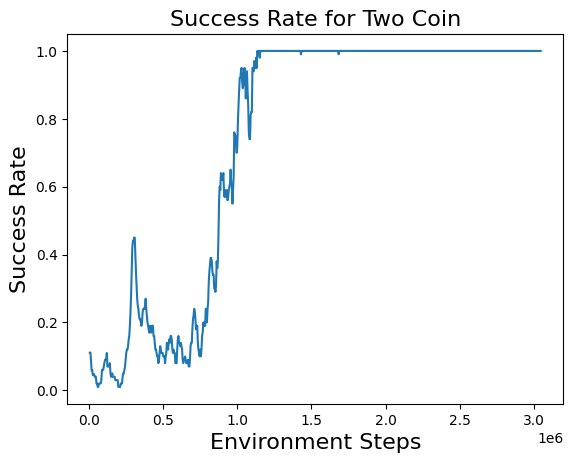}
    \caption{PPO-TransformerXL results on Searing Spotlights Environment. Success corresponds with the agent collecting both coins.}
    \label{fig:train_results}
\end{figure}

\begin{table}
    \centering
    \begin{tabular}{p{3in}l}
        \toprule
        \textbf{Parameter} & \textbf{Two-Coins Env}\\ 
        \midrule
        Max Episode Length           & 400           \\
        Agent Scale                  & 0.25          \\
        Agent Speed                  & 5             \\
        Agent Always Visible         & True          \\
        Agent Health                 & 10            \\
        Sample Agent Position        & True          \\
        Coin Scale                   & 0.375         \\
        Coin Show Duration           & n/a           \\
        Coin Always Visible          & True          \\
        Steps per Coin               & n/a           \\
        No. Initial Spotlight Spawns & 3             \\
        Spotlight Spawn Interval     & 30            \\
        Spotlight Radius             & (7.5, 13.75)  \\
        Spotlight Speed              & (0.008, 0.010)\\
        Spotlight Damage             & 0.05          \\
        Light Dim Off Duration       & 6             \\
        Light Threshold              & 255           \\
        Show Visual Feedback         & False         \\
        Render Background Black      & False         \\
        Hide Checkered Background    & False         \\
        Show Last Action             & False         \\
        Show Last Positive Reward    & False         \\
        Reward Inside Spotlight      & -0.05         \\
        Reward Outside Spotlight     & 0             \\
        Reward Death                 & 0             \\
        Reward Coin                  & 1             \\
        Reward Exit                  & n/a           \\
        Reward Per Step              & -0.005        \\
        \bottomrule
    \end{tabular}
    \caption{Environment parameters.}
\label{table:env_params}
\end{table}

\end{document}